\documentclass{article}
\usepackage{amsmath,amssymb,amsthm,color,enumitem,fancyhdr,float,fullpage,graphics,graphicx,longtable,mathrsfs,multirow,pgf,setspace,url,mathtools}
\usepackage{wrapfig}
\usepackage{tikz} 
\usetikzlibrary{positioning}
\usepackage[font=small,labelfont=bf]{caption}
\usepackage[subrefformat=parens]{subcaption}
\usepackage[toc,page]{appendix}
\usepackage{hyperref} 
\hypersetup{colorlinks=true,linkcolor=blue,citecolor=red}

\usepackage{cleveref}

\allowdisplaybreaks  
\usetikzlibrary{arrows}

\def\bold{\boldsymbol}

\def\|{\\ \smallskip}

\newcommand{\bn}{\mbox{$\mathbb{N}$}}

\newcommand{\bz}{\mbox{$\mathbb{Z}$}}

\newcommand{\bv}{\mbox{$  V$}}

\newcommand{\be}{\mbox{$  E$}}

\title{Graphical proof of Ginibre's inequality}
 \author{Yuki Tokushige\thanks{Technical University of Braunschweig, Institut f\"{u}r Mathematische Stochastik,
38106 Universit\"atsplatz 2 Braunschweig, Germany
 {\tt yuki.tokushige@tu-braunschweig.de}}}

\newtheorem{lem}{Lemma}[section]
\newtheorem{prop}[lem]{Proposition}

\newtheorem{thm}{Theorem}
\newtheorem{defi}[lem]{Definition}

\numberwithin{equation}{section}
\newcommand\blfootnote[1]{%
  \begingroup
  \renewcommand\thefootnote{}\footnote{#1}%
  \addtocounter{footnote}{-1}%
  \endgroup
}

\makeatletter
\def\slashedarrowfill@#1#2#3#4#5{%
  $\m@th\thickmuskip0mu\medmuskip\thickmuskip\thinmuskip\thickmuskip
   \relax#5#1\mkern-7mu%
   \cleaders\hbox{$#5\mkern-2mu#2\mkern-2mu$}\hfill
   \mathclap{#3}\mathclap{#2}%
   \cleaders\hbox{$#5\mkern-2mu#2\mkern-2mu$}\hfill
   \mkern-7mu#4$%
}

\def\rightslashedarrowfilla@{%
  \slashedarrowfill@\relbar\relbar{\raisebox{1.2pt}{$\scriptscriptstyle\diagup$}}\rightarrow}
\newcommand\xslashedrightarrowa[2][]{%
  \ext@arrow 0055{\rightslashedarrowfilla@}{#1}{#2}}

\def\rightslashedarrowfillb@{%
  \slashedarrowfill@\relbar\relbar/\rightarrow}
\newcommand\xslashedrightarrowb[2][]{%
  \ext@arrow 0055{\rightslashedarrowfillb@}{#1}{#2}}

\def\rightslashedarrowfillc@{%
  \slashedarrowfill@\relbar\relbar{\raisebox{.12em}{\tiny/}}\rightarrow}
\newcommand\xslashedrightarrowc[2][]{%
  \ext@arrow 0055{\rightslashedarrowfillc@}{#1}{#2}}

\pgfdeclareshape{slash underlined}
{
  \inheritsavedanchors[from=rectangle] % this is nearly a circle
  \inheritanchorborder[from=rectangle]
  \inheritanchor[from=rectangle]{north}
  \inheritanchor[from=rectangle]{north west}
  \inheritanchor[from=rectangle]{north east}
  \inheritanchor[from=rectangle]{center}
  \inheritanchor[from=rectangle]{west}
  \inheritanchor[from=rectangle]{east}
  \inheritanchor[from=rectangle]{mid}
  \inheritanchor[from=rectangle]{mid west}
  \inheritanchor[from=rectangle]{mid east}
  \inheritanchor[from=rectangle]{base}
  \inheritanchor[from=rectangle]{base west}
  \inheritanchor[from=rectangle]{base east}
  \inheritanchor[from=rectangle]{south}
  \inheritanchor[from=rectangle]{south west}
  \inheritanchor[from=rectangle]{south east}
  \inheritanchorborder[from=rectangle]
  \foregroundpath{
    \southwest \pgf@xa=\pgf@x \pgf@ya=\pgf@y
    \northeast \pgf@xb=\pgf@x \pgf@yb=\pgf@y
    \pgf@xc=\pgf@xa
    \advance\pgf@xc by .5\pgf@xb
    \pgf@yc=\pgf@ya
    \advance\pgf@xc by -1.3pt
    \advance\pgf@yc by -1.8pt
    \pgfpathmoveto{\pgfqpoint{\pgf@xc}{\pgf@yc}}
    \advance\pgf@xc by  2.6pt
    \advance\pgf@yc by  3.6pt
    \pgfpathlineto{\pgfqpoint{\pgf@xc}{\pgf@yc}}
    \pgfpathmoveto{\pgfqpoint{\pgf@xa}{\pgf@ya}}
    \pgfpathlineto{\pgfqpoint{\pgf@xb}{\pgf@ya}}
 }
}
\tikzset{nomorepostaction/.code=\let\tikz@postactions\pgfutil@empty}

\makeatother

\begin{document}
\maketitle
%\blfootnote{2010 {\it Mathematics Subject Classification}.  }
%\blfootnote{{\it Key words and phrases}. }
\begin{abstract}
 In this short note, we will give a new combinatorial proof of Ginibre's inequality for XY models. Our proof is based on multigraph representations introduced by van Engelenburg-Lis (2023) and a new combinatorial bijection.% which can be viewed as a combinatorial interpretation of Ginibre's square trick.
\end{abstract}


\section{Introduction}
For a finite graph $\Lambda=(V,E)$ and coupling constants $J=(J_{xy}):E\to(0,\infty)$, 
we consider the following Hamiltonian of the XY model: for ${\bold \sigma}=(\sigma_x)_{x\in V}\in(S^1)^V$, 
$$H({\bold \sigma}):=-\sum_{xy\in E}J_{xy}\sigma_x\cdot\sigma_y=-\sum_{xy\in E}\frac{J_{xy}}{2}(\sigma_x\bar{\sigma}_y+\bar{\sigma}_x\sigma_y)=-\sum_{\vec{xy}\in \vec{E}}\frac{J_{xy}}{2} \sigma_x\bar{\sigma}_y,$$
 where\ $S^1=\{z\in\mathbb{C}:|z|=1\}$) and $\vec{E}$ denotes the set of oriented edges of $\Lambda$.
The corresponding Gibbs measure is defined by
\begin{align*}
d\mu_{J,\Lambda}({\bold \sigma}):=\frac{1}{\mathcal{Z}_{J,\Lambda}}e^{-H({\bold \sigma})}d{\bold \sigma},
\end{align*}
where $d{\bold \sigma}=\prod_{x\in V}d\sigma_x$ is the uniform probability measure on $(S^1)^V$.
The normalizing constant $\mathcal{Z}_{J,\Lambda}:=\int_{(S^1)^V}e^{-H({\bold \sigma})}d{\bold \sigma}$ is called the partition function of the model, and by definition $\mu_{J,\Lambda}$
 is a probability measure on $(S^1)^V$. We will use a usual convention to denote the expectation with respect to $\mu_{J,\Lambda}$ by $\langle\cdot\rangle_{J,\Lambda}$.\par
 The XY models on $\mathbb{Z}^2$ are known to exhibit very exotic phase transitions. This was originally proved by Fr\"ohlich-Spencer \cite{FS81}. Recently, van Engelenburg-Lis \cite{XYEL23} and Aizenman-Harel-Peled-Schapiro \cite{AHPS} gave new proofs of this result using the relation between XY models and height models.
See \cite{GS1, GS2, LAMBIJ, LAMDIC} for more recent related works. \par
In order to state the result of this article, we introduce several notations.
For any function $\varphi:V\to\mathbb{Z}$, we will write
$${\bold \sigma}^{\varphi}:=\prod_{x\in V}\sigma_x^{\varphi_x}.$$
Furthermore, for ${\bold \sigma}=(\sigma_x),{\bold \sigma}'=(\sigma_x')\in(S^1)^{V}$
we define $\boldsymbol{\theta}=(\theta_x)_{x\in V},\boldsymbol{\theta}'=(\theta'_x)_{x\in V}$ as elements of
$[0,2\pi)^{ V }$ satisfying 
$$\sigma_x=e^{{\bold i}\theta_x}, \  \sigma'_x=e^{{\bold i}\theta'_x}.$$
We then write
$$\varphi\cdot{\bold \theta}:=\sum_{x\in V}\varphi_x\theta_x,\ \ \varphi\cdot{\bold \theta}':=\sum_{x\in V}\varphi_x\theta_x'.$$
We finally state the claim we will show in this article.
\begin{thm}[\cite{ginibre1970}] \label{thm:aaa}
For $\varphi,\psi:V\to\mathbb{Z}$ with $\sum_{v\in V}\varphi_v=\sum_{v\in V}\psi_v=0$, it holds that
$$\langle \sigma^{\varphi+\psi}\rangle_{J,\Lambda}+\langle \sigma^{\varphi-\psi}\rangle_{J,\Lambda}\geq
2\langle\sigma^{\varphi}\rangle_{J,\Lambda}\langle\sigma^{\psi}\rangle_{J,\Lambda}.$$
Since $\sigma^{\varphi}+\sigma^{-\varphi}=2\cos(\varphi\cdot{\bold \theta})$,
this is equivalent to
$$\langle \cos(\varphi\cdot{\bold \theta})  \cos(\psi\cdot{\bold \theta})\rangle_{J,\Lambda}\geq\langle\cos(\varphi\cdot{\bold \theta})\rangle_{J,\Lambda}\langle\cos(\psi\cdot{\bold \theta})\rangle_{J,\Lambda}.$$
\end{thm}
The result is not new. It is nothing but the celebrated Ginibre's inequality shown in \cite{ginibre1970}. The aim of this article is to give a new combinatorial proof of Theorem \ref{thm:aaa}, which is based on multigraph representations introduced in \cite{XYEL23ar}. % where the authors stated that their representations do not immediately give a proof of Theorem \ref{thm:aaa}. 
 The novelty of this work is a new bijection that will play an important role in our proof. 
 %This bijection can be interpreted as an expression of a hidden combinatorial structure of the famous Ginibre's trick \cite{ginibre1970} from which Theorem \ref{thm:aaa} was originally derived. 
 It seems to the author an interesting problem to investigate the possibility of applying our combinatorial idea to obtain new correlation inequalities of XY models.% other than Ginibre's inequality.
 \par
This paper is organized as follows: in Section \ref{sec2}, we will briefly review random current and multigraph representations for the XY model. In Section \ref{3}, we will introduce a new bijection and prove Theorem \ref{thm:aaa}. %Finally in Section \ref{4}, we will explain how the bijection introduced in Section \ref{3} is related to Ginibre's trick. 

\section*{Acknowledgements}
The author would like to thank Diederik van Engelenburg for reading an earlier version of the manuscript and giving him suggestions which helped simplify notation and the proof. He would also like to thank Piet Lammers for letting him know multigraph representations introduced in \cite{XYEL23ar}. He is grateful to Takashi Hara and Akira Sakai for their encouragements. 
 He also thanks Tyler Helmuth for his comments.

\section{Graphical representations and Ginibre's inequality}\label{sec2}
\subsection{Random current representations}
In this article, a $\mathbb{Z}_{\geq0}$-valued function ${\bf n}$ defined on the set of oriented edges $\vec{E}$ will be called a {\it current}. For each current ${\bf n}$, we define its {\it amplitude} $|{\bf n}|:E\to\mathbb{Z}_{\geq0}$ by
$$|{\bf n}|_{xy}:={\bf n}_{\vec{xy}}+{\bf n}_{\vec{yx}}.$$
We define the {\it source function} $\partial{\bf n}$ of ${\bf n}$, which is a $\mathbb{Z}$-valued function on $V$, by
$$\partial{\bf n}_x:=\sum_{y\in V:y\sim x}({\bf n}_{\vec{xy}}-{\bf n}_{\vec{yx}}).$$
These functions naturally arise in the study of XY models. To see this, we will compute the partition function $\mathcal{Z}_{J,\Lambda}$ of an XY model with coupling constants $(J_{xy})$ defined on a domain $\Lambda$. Let 
$d{\bold \sigma}=\prod_{x\in V}d\sigma_x$ be the uniform probability measure on $[0,2\pi)^{V}$ and define the weight $w_J({\bf n})$ of a current ${\bf n}$ by
$$w_J({\bf n}):=\prod_{\vec{xy}\in\vec{E}}\dfrac{\left(\frac{ J_{xy} }{2}\right)^{ {\bf n}_{\vec{xy}} }  }{{\bf n}_{\vec{xy}}!  }.$$ 
 Then, 
\begin{align*}
\mathcal{Z}_{J,\Lambda}&:=%\int_{[0,2\pi)^{ V}}\exp\left[\sum_{xy\in  E}J_{xy}\sigma_x\cdot\sigma_y\right]d{\bold \sigma}\\
%&=
\int_{[0,2\pi)^{V}}\exp\left[\sum_{xy\in E}\dfrac{J_{xy}}{2}(\sigma_x\bar{\sigma}_y+\bar{\sigma}_x\sigma_y)\right]d{\bold \sigma}\\
&=\int_{[0,2\pi)^{V}}\exp\left[\sum_{\vec{xy}\in \vec{E}}\dfrac{J_{xy}}{2}\sigma_x\bar{\sigma}_y\right]d{\bold \sigma}\\
&=\int_{[0,2\pi)^{V}}\prod_{\vec{xy}\in \vec{E}}\left(\sum_{{\bf n}_{\vec{xy}}=0}^{\infty}
\dfrac
{
\left[\frac{J_{xy}}{2}\sigma_x\bar{\sigma}_y\right]^{{\bf n}_{\vec{xy}}}
}
{
{\bf n}_{\vec{xy}}!
}
\right)d{\bold \sigma}\\
&=\sum_{{\bf n}=({\bf n}_{\vec{xy}})_{\vec{xy}\in\vec{E}}}w_J({\bf n})
\int_{[0,2\pi)^{V}}\prod_{\vec{xy}\in \vec{E}}
\left(\sigma_x\bar{\sigma}_y\right)^{ {\bf n}_{\vec{xy}} }d{\bold \sigma}\\
&=\sum_{{\bf n}=({\bf n}_{\vec{xy}})_{\vec{xy}\in\vec{E}}}w_J({\bf n})
\prod_{x\in V }\int_{ [0,2\pi) }
\sigma_x^{ \partial{\bf n}_x }d \sigma_x\\
&=\sum_{\partial{\bf n}=0}w_J({\bf n}).
\end{align*}
A straightforward generalization yields the following well-known equality, which gives random current representations of  correlation functions. It is well-known that analogous representations for Ising models are very powerful tools to study their critical phenomena \cite{Aizenman82, AG, ABF, AF, A-DC-S, DC-T1, DC-T2, A-DC}.
\begin{prop}\label{prop:rcr}
For any function $\varphi:V\to\mathbb{Z}$ we have
$$
\langle\sigma^{\varphi}\rangle_{J,\Lambda}
=\frac{1}{\mathcal{Z}_{J,\Lambda}}\sum_{ \partial{\bf n}=-\varphi}w_J({\bf n})
=\frac{1}{\mathcal{Z}_{J,\Lambda}}\sum_{ \partial{\bf n}=\varphi}w_J({\bf n}).
$$
The second equality follows from the fact that $w_J$ is invariant with respect to the map ${\bf n}_{\vec{xy}}\mapsto{\bf n}_{\vec{yx}}$.
\end{prop}

\subsection{Multigraph representations and Ginibre's inequality}
In this subsection, we will relate a random current representation that appeared in Proposition \ref{prop:rcr} to a counting of directed multigraphs involving two colors, say, red and blue. This idea was used in \cite{ XYEL23ar, XYEL23}, where the authors analyzed the two dimensional XY model.
We then reformulate Ginibre's inequality using this multigraph representation. \par
Let $\varphi,\psi: V\mapsto\mathbb{Z}$. By Proposition \ref{prop:rcr}, we have
\begin{align}\label{eq:1}
\mathcal{Z}_{J,\Lambda}^2\langle \sigma^{\varphi} \rangle_{J,\Lambda}
&=\sum_{  \substack{\partial{\bf n}=\varphi\\ \ \partial{\bf m}=0   } }w_J({\bf n})w_J({\bf m})\nonumber\\
&=\sum_{ {\bf N }:E\to\mathbb{Z}_{\geq0}  }  \sum_{  \substack{\partial{\bf n}=\varphi\\ \ \partial{\bf m}=0 \\ |{\bf n}+{\bf m}|={\bf N}   } } 
\left[\prod_{ xy\in E} \dfrac{\left(\frac{J_{xy}}{2}\right)^{{\bf N}_{xy}}}{{\bf N}_{xy}!}  \right]
 \left[\prod_{ xy\in E} \dfrac{ {\bf N}_{xy}!}{ {\bf n}_{\vec{xy}}! {\bf n}_{\vec{yx}}! {\bf m}_{\vec{xy}}! {\bf m}_{\vec{yx}}! }  \right].
\end{align}
In the same way, we get
\begin{align}\label{eq:2}
\mathcal{Z}_{J,\Lambda}^2\langle \sigma^{\varphi} \rangle_{J,\Lambda}\langle \sigma^{\psi} \rangle_{J,\Lambda}
&=\sum_{  \substack{\partial{\bf n}=\varphi\\ \ \partial{\bf m}=\psi    } }w_J({\bf n})w_J({\bf m})\nonumber\\
&=\sum_{ {\bf N }:E\to\mathbb{Z}_{\geq0}  }  \sum_{  \substack{\partial{\bf n}=\varphi\\ \ \partial{\bf m}=\psi \\ |{\bf n}+{\bf m}|={\bf N}   } } 
\left[\prod_{ xy\in E} \dfrac{\left(\frac{J_{xy}}{2}\right)^{{\bf N}_{xy}}}{{\bf N}_{xy}!}  \right]
 \left[\prod_{ xy\in E} \dfrac{ {\bf N}_{xy}!}{ {\bf n}_{\vec{xy}}! {\bf n}_{\vec{yx}}! {\bf m}_{\vec{xy}}! {\bf m}_{\vec{yx}}! }  \right].
\end{align}
Let us consider a combinatorial interpretation of 
\begin{align}\label{formula:1}
\prod_{ xy\in E} \dfrac{ {\bf N}_{xy}!}{ {\bf n}_{\vec{xy}}! {\bf n}_{\vec{yx}}! {\bf m}_{\vec{xy}}! {\bf m}_{\vec{yx}}!}.
\end{align} 
To begin with, we define $\mathbb{G}_{\bf N}$ to be an unoriented multigraph which, for each $xy\in E$, has ${\bf N}_{xy}$ distinguishable edges connecting $x$ and $y$. We then interpret \eqref{formula:1} as the number of ways to paint edges of $\mathbb{G}_{\bf N}$ in red and blue and give them orientations in such a way that:
\begin{itemize}
\item for each $xy\in E$ there are $|{\bf n}|_{xy}$ (resp. $|{\bf m}|_{xy}$) red (resp. blue) edges.
\item For each $xy\in E$, among $|{\bf n}|_{xy}$ red edges connecting $x$ and $y$, there are ${\bf n}_{\vec{xy}}$ edges oriented from $x$ to $y$, and ${\bf n}_{\vec{yx}}$ edges oriented from $y$ to $x$. The same property holds for blue edges when we replace ${\bf n}$ with ${\bf m}$.
\end{itemize}
In what follows, for vertices $x,y\in\mathbb{G}_{\bf N}$ we write ${\bf r}_{\vec{xy}}$ (resp. ${\bf b}_{\vec{xy}}$) for the number of red (resp. blue) edges in $\mathbb{G}_{\bf N}$ (with colors and orientation) which are oriented from $x$ to $y$.
Then
the following equalities follow from the above observations and \eqref{eq:1}, \eqref{eq:2}. 
\begin{align}\label{eq:3}
\mathcal{Z}_{J,\Lambda}^2\langle \sigma^{\varphi} \rangle_{J,\Lambda}
&=\sum_{ {\bf N }:E\to\mathbb{Z}_{\geq0}  }  
%\sum_{  \substack{\partial{\bf n}=0\\ \ \partial{\bf m}=-\psi \\ |{\bf n}+{\bf m}|={\bf N}   } } 
\left[\prod_{ xy\in E} \dfrac{\left(\frac{J_{xy}}{2}\right)^{{\bf N}_{xy}}}{{\bf N}_{xy}!}  \right]
\cdot
\#\{\partial{\bf r}=\varphi, \partial{\bf b}=0,\ |{\bf r}+{\bf b}|={\bf N}\},\nonumber
\intertext{and} 
\mathcal{Z}_{J,\Lambda}^2\langle \sigma^{\varphi} \rangle_{J,\Lambda}\langle \sigma^{\psi} \rangle_{J,\Lambda}
&=\sum_{ {\bf N }:E\to\mathbb{Z}_{\geq0}  }  
\left[\prod_{ xy\in E} \dfrac{\left(\frac{J_{xy}}{2}\right)^{{\bf N}_{xy}}}{{\bf N}_{xy}!}  \right]
\cdot
\#\{\partial{\bf r}=\varphi, \partial{\bf b}=\psi,\ |{\bf r}+{\bf b}|={\bf N}\}.
\end{align}
For brevity of notation we will write $\{\partial{\bf r}=\varphi, \partial{\bf b}=\psi\}_{\bf N}$ for $\{\partial{\bf r}=\varphi, \partial{\bf b}=\psi,\ |{\bf r}+{\bf b}|={\bf N}\}$, and in the next section we will adopt the same rule for sets of oriented multigraphs involving only one color.\par
Thus, Theorem \ref{thm:aaa} will follow from
\begin{align}\label{goal}
\#\{\partial{\bf r}=\varphi+\psi, \partial{\bf b}=0\}_{\bf N}+\#\{\partial{\bf r}=\varphi-\psi, \partial{\bf b}=0\}_{\bf N}
\geq 2\#\{\partial{\bf r}=\varphi, \partial{\bf b}=\psi\}_{\bf N}
\end{align}
for any ${\bf N}:E\to\mathbb{Z}_{\geq0}$, where $\varphi,\psi:V\to\mathbb{Z}$ are functions with mean zero and $\psi\neq0$.
In Section \ref{3}, we will prove \eqref{goal}.

\section{Combinatorial proof of Ginibre's inequality}\label{3}

We will prove \eqref{goal} by applying a certain bijection to 
$$\{\partial{\bf r}=\varphi+\psi, \partial{\bf b}=0\}_{\bf N}\cup\{\partial{\bf r}=\varphi-\psi, \partial{\bf b}=0\}_{\bf N}\ \ \text{and}\ \ \{\partial{\bf r}=\varphi, \partial{\bf b}=\psi\}_{\bf N}.$$
%The bijection we will introduce in this section is inspired by the famous trick used in \cite{ginibre1970}, from which Ginibre's inequality was derived. In Section \ref{4}, we will expand on this matter by showing computations that underlie combinatorial arguments we will present here. \par
%One identity that our bijective combinatorial argument yields is
%\begin{align}\label{eq:square}
%\#\{\partial{\bf r}=\partial{\bf b}=0 ; {\bf N}\}=\left( \sum_{   \substack{ {\bf r}:\ |{\bf r}|={\bf N}\\\partial{\bf r}=0 }   }\binom{\bf N}{\bf r}\right)^2,
%\end{align}
%which does not seem obvious.

\subsection{Proof of \eqref{goal} }
For any $f,g:V\to\mathbb{Z}$ with $\sum_{v\in V}f(v)=\sum_{v\in V}g(v)=0$, we will construct a bijection 
\begin{align}\label{bij}
F:\{\partial{\bf r}=f, \partial{\bf b}=g\}_{\bf N} \to \{\partial{\bf r}=f+g\}_{\bf N}\times\{\partial{\bf b}=-f+g\}_{\bf N}.
\end{align}
Once we have this bijection, it immediately follows that
$$\#\{\partial{\bf r}=f, \partial{\bf b}=g\}_{\bf N}=\# \{\partial{\bf r}=f+g\}_{\bf N}\cdot\#\{\partial{\bf b}=-f+g\}_{\bf N}.$$
Then we obtain the inequality \eqref{goal} as follows: 
\begin{align*}
&\ \ \ \#\{\partial{\bf r}=\varphi\pm\psi, \partial{\bf b}=0\}_{\bf N}-
 2\#\{\partial{\bf r}=\varphi, \partial{\bf b}=\psi\}_{\bf N}\nonumber\\
 &=\#\{\partial{\bf r}=\varphi+\psi\}_{\bf N}\cdot\#\{\partial{\bf b}=-\varphi-\psi\}_{\bf N}+\#\{\partial{\bf r}=\varphi-\psi\}_{\bf N}\cdot\#\{\partial{\bf b}=-\varphi+\psi\}_{\bf N}\nonumber\\
 &\hspace{65mm}-2\#\{\partial{\bf r}=\varphi+\psi\}_{\bf N}\cdot\#\{\partial{\bf b}=-\varphi+\psi\}_{\bf N}\nonumber\\
 &=\left(\#\{\partial{\bf r}=\varphi+\psi\}_{\bf N}\right)^2+\left(\#\{\partial{\bf r}=\varphi-\psi\}_{\bf N}\right)^2
 -2\#\{\partial{\bf r}=\varphi+\psi\}_{\bf N}\cdot\#\{\partial{\bf r}=\varphi-\psi\}_{\bf N}\nonumber\\
 &=\left(\#\{\partial{\bf r}=\varphi+\psi\}_{\bf N}-\#\{\partial{\bf r}=\varphi-\psi\}_{\bf N}\right)^2\geq0.
\end{align*}

\subsection{Construction of a bijection \eqref{bij}}
Suppose that we have an element $\omega\in\{\partial{\bf r}=f, \partial{\bf b}=g\}_{\bf N}$.
We define an element 
$$(\omega_{\tt R},\omega_{\tt B})\in \{\partial{\bf r}=f+g\}_{\bf N}\times\{\partial{\bf b}=-f+g\}_{\bf N}$$
 that corresponds to $\omega$ in the following way: we define $\omega_{\tt R}\in\{\partial{\bf r}=f+g\}_{\bf N}$ simply by painting all the edges in $\omega$ in red. We then define $\omega_{\tt B}\in\{\partial{\bf b}=-f+g\}_{\bf N}$ as follows: we 
 retain blue edges in $\omega$ and reverse orientations of red edges and paint all of them in blue. See Figure \ref{fig:aaa}.
%Since both of red and blue configuration in $\omega$ are sourceless,
It is straightforward to see that this defines a well-defined map 
\begin{align*}
F:\{\partial{\bf r}=f, \partial{\bf b}=g\}_{\bf N} &\to \{\partial{\bf r}=f+g\}_{\bf N}\times\{\partial{\bf b}=-f+g\}_{\bf N}\\
\omega&\mapsto (\omega_{\tt R},\omega_{\tt B}).
\end{align*}

\begin{figure}[h]
 \begin{minipage}[b]{0.33\linewidth}
    \centering
    \includegraphics[width=0.8\textwidth]{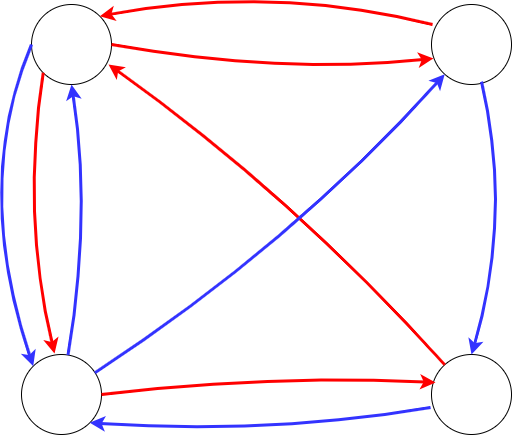}
    \subcaption{$\omega\in\{\partial{\bf r}=\partial{\bf b}=0\}_{\bf N}$}
  \end{minipage}
  \begin{minipage}[b]{0.33\linewidth}
    \centering
    \includegraphics[width=0.8\textwidth]{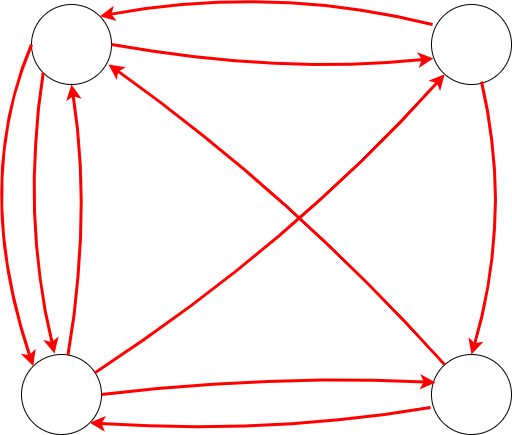}
    \subcaption{$\omega_{\tt R}\in\{\partial{\bf r}=0\}_{\bf N}$}
  \end{minipage}
   \begin{minipage}[b]{0.33\linewidth}
    \centering
    \includegraphics[width=0.8\textwidth]{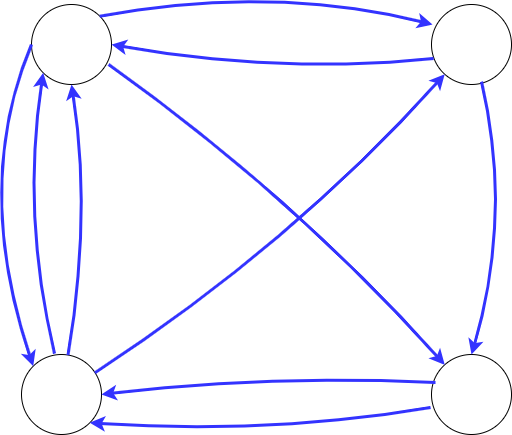}
    \subcaption{$\omega_{\tt B}\in\{\partial{\bf b}=0\}_{\bf N}$}
  \end{minipage}
    \caption{A bijection between $\{\partial{\bf r}=\partial{\bf b}=0\}_{\bf N}\ \text{and}\ \{\partial{\bf r}=0\}_{\bf N}\times\{\partial{\bf b}=0\}_{\bf N}$}\label{fig:aaa}
\end{figure}

We next describe the map in the opposite direction. Suppose that we have an element 
$(\omega_{\tt R},\omega_{\tt B})\in\{\partial{\bf r}=f+g\}_{\bf N}\times\{\partial{\bf b}=-f+g\}_{\bf N}$.
 We then define a new configuration $\tilde{\omega}$ in the following way:
we first obtain an oriented colorless configuration $\xi$ on $\mathbb{G}_{\bf N}$ simply by forgetting the color of $\omega_{\tt R}$. The orientation of $\tilde{\omega}$ is defined by $\xi$.
We then paint edges in $\xi$ according to the following rule: for each oriented edge $\vec{e}$ in $\xi$,
let $e$ be the corresponding unoriented edge in $\mathbb{G}_{\bf N}$. We then look at two oriented edges in $\omega_{\tt R}$ and ${\omega}_{\tt B}$ each of which corresponds to $e$. We paint $\vec{e}$ in blue if they have the same direction, and paint $\vec{e}$ in red otherwise. \par
That this operation defines a well-defined map from 
$\{\partial{\bf r}=f+g\}_{\bf N}\times\{\partial{\bf b}=-f+g\}_{\bf N}$ to $\{\partial{\bf r}=f, \partial{\bf b}=g\}_{\bf N}$
 is verified by the following easy observation: let ${\bf r}(\tilde{\omega})$ (resp. ${\bf b}(\tilde{\omega})$) be the red (resp. blue) configuration in $\tilde{\omega}$. Then by definition, we have
 \begin{align*}
 \partial({\bf r}(\tilde{\omega}))+\partial({\bf b}(\tilde{\omega}))&=f+g,\\
 -\partial({\bf r}(\tilde{\omega}))+\partial({\bf b}(\tilde{\omega}))&=-f+g,
 \end{align*}
 which implies that $\tilde{\omega}$ is in $\{\partial{\bf r}=f, \partial{\bf b}=g\}_{\bf N}$.
Now it is straightforward to verify $F(\tilde{\omega})=(\omega_{\tt R},\omega_{\tt B})$, which implies that
$$F:\{\partial{\bf r}=f, \partial{\bf b}=g\}_{\bf N} \to \{\partial{\bf r}=f+g\}_{\bf N}\times\{\partial{\bf b}=-f+g\}_{\bf N}$$
is indeed a bijection.

\medskip
{\sf Data Availability}\ \ Data sharing not applicable to this article as no datasets were generated or analyzed during the current study.\\\\
\medskip
{\sf Conflict of Interest}\ \ There is no conflict of interest.

\if0
\begin{align}
&\ \ \int_{([0,2\pi)\times [0,2\pi))^{  V  }    }
\prod_{xy\in  E}
\left[\cos(\theta_x-\theta_y)+\cos(\theta_x'-\theta_y')\right]^{{\bf N}_{xy}}
d{\bold \theta} d{\bold \theta}'\nonumber\\
&=2^{\sum {\bf N}_{xy}}
\left( \int_{((-2\pi,2\pi])^{  V } }\prod_{xy\in  E}\left[\cos\frac{\eta_{xy}}{2}\right]^{ {\bf N}_{xy}}\dfrac{d\bold{\eta}}{(4\pi)^{\#  V}}\right)^2\nonumber\\
&=2^{-\sum {\bf N}_{xy}}
\left(\{\partial{\bf r}=0\}_{\bf N}   \right)^2.
\end{align}

Therefore, by \eqref{eq:rewrite} we have
\begin{align*}
&\ \ \#\{\partial{\bf r}=\pm(\delta_x-\delta_y)\pm(\delta_z-\delta_w),\ \partial{\bf b}=0\}_{\bf N}-
\#\{\partial{\bf r}=\pm(\delta_x-\delta_y),\ \partial{\bf b}=\pm(\delta_z-\delta_w)\}_{\bf N}\\
&=\frac{1}{2}\left(\sum_{ \substack{ {\bf r}:\ |{\bf r}|={\bf N}\\ \partial{\bf r}= \pm(\delta_x-\delta_y)\pm(\delta_z-\delta_w) }     }\binom{\bf N}{\bf r}-\sum_{ \substack{ {\bf r}:\ |{\bf r}|={\bf N}\\ \partial{\bf r}= \pm(\delta_x-\delta_y)\mp(\delta_z-\delta_w) }     }\binom{\bf N}{\bf r}\right)^2\geq0
\intertext{and}
&\ \ \#\{\partial{\bf r}=\partial{\bf b}=0 ; {\bf N}\}=\left( \sum_{   \substack{ {\bf r}:\ |{\bf r}|={\bf N}\\\partial{\bf r}=0 }   }\binom{\bf N}{\bf r}\right)^2.
\end{align*}
\fi

\bibliographystyle{plain}
\bibliography{BIB}

\end{document}